\documentclass{amsart}
\RequirePackage[utf8]{inputenc}
\usepackage{amsmath,amsthm,amsfonts,amssymb,amscd,amsbsy,multirow,tikz,hyperref}
\usepackage[all]{xy}
\usepackage{graphicx,pgf,caption,subcaption}
\usepackage{enumerate}
\usepackage{mathdots}
\usepackage{dsfont}
\usepackage{cleveref}
\usepackage{stmaryrd}

\usepackage[colorinlistoftodos,disable]{todonotes}

\newcommand{\F}{\mathcal{F}}

\newcommand{\R}{\mathds R}

\newcommand{\conv}{\operatorname{conv}}

\newcommand{\id}{\operatorname{Id}}

\newcommand{\SO}{\mathsf{SO}}
\renewcommand{\O}{\mathsf O}

\allowdisplaybreaks

\hypersetup{
    pdftoolbar=true,
    pdfmenubar=true,
    pdffitwindow=false,
    pdfstartview={FitH},
    pdftitle={A geometric take on Kostant's Convexity Theorem},
    pdfauthor={Ricardo A. E. Mendes},
    pdfsubject={},
    pdfkeywords={}
    pdfnewwindow=true,
    colorlinks=true, %set this false for printable version
    linkcolor=blue,
    citecolor=blue,
    urlcolor=black,
}

\newtheorem{theorem}{Theorem}[]
\newtheorem{lemma}[theorem]{Lemma}
\newtheorem{proposition}[theorem]{Proposition}
\newtheorem{corollary}[theorem]{Corollary}

\newtheorem{mainthm}{\sc Theorem}

\newtheorem{maincor}[mainthm]{\sc Corollary}

\newtheorem{question}[theorem]{Question}

\theoremstyle{definition}
\newtheorem{definition}[theorem]{Definition}

\theoremstyle{remark}
\newtheorem{remark}[theorem]{Remark}

\newtheorem{example}[theorem]{Example}

\title{A geometric take on Kostant's Convexity Theorem}
\author[R. A. E. Mendes]{Ricardo A. E. Mendes}
\address{University of Oklahoma\newline
\indent Department of Mathematics\newline
\indent 601 Elm Ave\newline
\indent Norman, OK, 73019-3103, USA}
\email{ricardo.mendes@ou.edu}

\thanks{The author has been supported by the NSF grant DMS-2005373}


\begin{document}

\begin{abstract}
Given a compact Lie group $G$ and an orthogonal $G$-representation $V$, we give a purely metric criterion for a closed subset of the orbit space $V/G$ to have convex pre-image in $V$. In fact, this also holds with the natural quotient map $V\to V/G$ replaced with an arbitrary submetry $V\to X$.

In this context, we introduce a notion of ``fat section'' which generalizes polar representations, representations of non-trivial copolarity, and isoparametric foliations.

We show that Kostant's Convexity Theorem partially generalizes from polar representations to submetries with a fat section, and give examples illustrating that it does not fully generalize to this situation.
\end{abstract}

\maketitle

\section{Introduction}

B. Kostant's celebrated ``Convexity Theorem'' \cite[Theorem 8.2]{Kostant73} can be phrased as follows:
\begin{theorem}[Kostant]
\label{T:Kostant}
Let $V$ be a real orthogonal representation of the compact group $G$, with connected orbits. Assume the representation is polar, with section $\Sigma\subset V$ and Weyl group $W$ acting on $\Sigma$. Then 
\[\pi_\Sigma(G\cdot v) = \conv (W\cdot v)\]
for all $v\in\Sigma$. Here $\pi_\Sigma$  denotes the orthogonal projection onto $\Sigma$, $\conv(\cdot)$ the convex hull, $G\cdot v$ the $G$-orbit through $v$, and $W\cdot v=(G\cdot v)\cap\Sigma$ the $W$-orbit through $v$.
\end{theorem}
Recall that the $G$-representation $V$ is called polar with section $\Sigma$ if $\Sigma$ is a linear subspace of $V$ that intersects all $G$-orbits orthogonally (see \cite{Dadok85}). The formulation above is equivalent to the original one because, up to having the same orbits, the class of polar representations with connected orbits coincides with the class of isotropy representations of symmetric spaces (see \cite[Proposition 6]{Dadok85}). Two notable special cases are the adjoint representation of a connected compact Lie group $G$ on its Lie algebra $V$, with $\Sigma$  the Lie algebra of a maximal torus and $W$ the usual Weyl group; and the Schur--Horn Theorem \todo{explain or give reference?} concerning the diagonal entries of symmetric matrices with a given set of eigenvalues (see \cite[page 150]{LRT99}).

An important consequence is that the study of $G$-invariant convex subsets of $V$ is in some sense reduced to the study of $W$-invariant convex subsets of $\Sigma$:
\begin{corollary}
\label{C:Kostant}
With assumptions and notations as in Theorem \ref{T:Kostant}:
\begin{enumerate}[(a)] 
\item For every $G$-invariant convex subset $K\subset V$,
\[\pi_\Sigma(K)=K\cap\Sigma.\]
\item The map $K\mapsto \pi_\Sigma(K)=K\cap\Sigma$ is a bijection between $G$-invariant convex subsets of $V$ and $W$-invariant convex subsets of $\Sigma$.
\end{enumerate}
\end{corollary}
%\todo{give proof? or cite Tim Kobert's thesis paper with Scheiderer (Cor. 3.4?)?}
For a proof of (a), see e.g. \cite[Corollary 3.4]{KS20}. The proof of (b) is straightforward (see proof of Theorem \ref{MT:fatsection} below for a generalization).
%\footnote{It suffices to show that $G\cdot T$ is convex in $V$, for every $T\subset \Sigma$ that is  $W$-invariant and convex. Let $K=\conv(G\cdot T)$. By (a), we have $K\cap\Sigma=\pi_\Sigma(K)$. But $\pi_\Sigma(K)=\conv(\pi_\Sigma(G\cdot T))=T$ by Theorem \ref{T:Kostant} and convexity of $T$. Thus $K\cap\Sigma=T$, which implies $K=G\cdot T$, that is, $G\cdot T$ is convex.}

In terms of quotient spaces, subsets of $V/G$ (resp. $\Sigma/W$) correspond to $G$-invariant subsets of $V$ (resp. $W$-invariant subsets of $\Sigma$). Thus Corollary \ref{C:Kostant}(b) states that the isometry $\Sigma /W\to V/G$ induced by the inclusion $\Sigma\to V$ preserves the collection of subsets that have convex pre-images in $\Sigma, V$. This is an immediate consequence of our first main result, which gives a \emph{purely metric} criterion for a closed subset $S\subset V/G$ to have convex pre-image in $V$, for \emph{any} (not necessarily polar) $G$-representation $V$. More generally, one can also replace the map $V\to V/G$ with any submetry, that is, any map $V\to X$ to a metric space $X$ that takes metric balls to metric balls of the same radius (see Subsection \ref{SS:submetries} below):

\begin{mainthm}
\label{MT:quotient}
Let $V$ be a finite-dimensional real vector space with an inner product, $X$ a metric space, and $\sigma\colon V\to X$ a submetry. Let $S\subset X$ be a closed subset, and denote by $\delta_S$
%$f\colon X\setminus S\to (0,\infty)$ 
the distance function to $S$. Then $\sigma^{-1}(S)$ is a convex subset of $V$ if and only if the map $\delta_S:\overline{X\setminus S}\to [0,\infty)$ is a submetry.

%Then the following are equivalent:
%\begin{enumerate}[(a)]
%\item $\sigma^{-1}(S)$ is a convex subset of $V$.
%\item $\|\nabla_x \delta_S\|=1$ for every $x\in X\setminus S$.
%\item The map $\delta_S:\overline{X\setminus S}\to [0,\infty)$ is a submetry.
%\end{enumerate}

%\[ \lim\sup_{y\to x} \frac{f(y)-f(x)}{d(y,x)}=1 \]
%for every $x\in X\setminus S$.
\end{mainthm}

%The gradient of the function $\delta_S$ above can be made precise in Alexandrov Geometry, see Subsection \ref{SS:Alexandrov}.

Theorem \ref{MT:quotient} fits in a general theme: the interplay between the geometry of the isometric $G$-action on $V$ (or, more generally, some Riemannian manifold) and the metric geometry of the orbit space $V/G$. This is a traditional topic explored extensively in the literature, for example in \cite{HK89, Grove02, GL14, GLLM19, Mendes20, GKW21}.

Going beyond the polar case, whenever a $G$-representation $V$ admits a ``reduction'', that is, a $G'$-representation $V'$ such that $V/G$ is isometric to $V'/G'$ (called ``quotient-equivalent'' in \cite{GL14}), Theorem \ref{MT:quotient} implies that there is a bijection between the respective closed invariant convex subsets. More generally:
\begin{maincor}
\label{MC:isometry}
Let $V, V'$ be  finite-dimensional real vector spaces with inner product, $X,X'$  metric spaces, and $\sigma\colon V\to X$ and $\sigma'\colon V'\to X'$ be submetries. Let $\varphi\colon X\to X'$ be an isometry, and $S\subset X$ be a closed subset. Then $\sigma^{-1}(S)$ is convex if and only if $(\sigma')^{-1}(\varphi(S))$ is convex. Thus $\varphi$ induces a bijection between closed convex $\sigma$-saturated subsets of $V$ and closed convex $\sigma'$-saturated subsets of $V'$.
\end{maincor}

Having generalized Corollary \ref{C:Kostant}(b), we investigate the extent to which Theorem \ref{MT:quotient} can be used to generalize Theorem \ref{T:Kostant} and Corollary \ref{C:Kostant}(a). Since these concern the orthogonal projection $\pi_\Sigma$, we need the ``reduction'' to be induced by a vector subspace $\Sigma\subset V$. More precisely, we introduce the following definition: Given a submetry $\sigma\colon  V\to X$, such that $\{0\}$ is a fiber, we call a subspace $\Sigma\subset V$ a \emph{fat section}\footnote{Term borrowed from \cite{Magata09,Magata10}.} for $\sigma$ if $\sigma|_\Sigma\colon \Sigma\to X$ is a submetry. Besides polar representations, this generalizes a number of other objects, including representations of non-trivial copolarity, principal isotropy group reductions, and isoparametric foliations (see Section \ref{S:fatsections}). Our second main result is: %summarizes the positive answers stemming from this investigation:
\begin{mainthm}
\label{MT:fatsection}
Let $\sigma\colon  V\to X$ be a submetry such that $\{0\}$ is a fiber, with fat section $\Sigma\subset V$. Denote by $\pi_\Sigma$ the orthogonal projection onto $\Sigma$. Then:
\begin{enumerate}[(a)]
\item For any $\sigma$-fiber $F\subset V$,
\[\conv(F)\cap\Sigma=\pi_\Sigma(\conv(F))=\conv(F\cap\Sigma).\]
In particular,
\[\pi_\Sigma(F)\subset \conv(F\cap\Sigma).\]
\item For every $\sigma$-saturated convex set $K$, 
\[\pi_\Sigma(K)=K\cap\Sigma.\]
\item The map $K\mapsto \pi_\Sigma(K)=K\cap\Sigma$ is a bijection between $\sigma$-saturated convex subsets of $V$ and $\sigma|_\Sigma$-saturated convex subsets of $\Sigma$.
\end{enumerate}
\end{mainthm}

Thus Corollary \ref{C:Kostant} holds in this more general situation. Specifically, Theorem \ref{MT:fatsection} (b) (respectively, (c)) is a direct generalization of Corollary \ref{C:Kostant} (a) (respectively, (b)). Moreover, Theorem \ref{MT:fatsection} (a) gives a direct generalization of one of the inclusions in the statement of Theorem \ref{T:Kostant}, namely ``$\pi_\Sigma(G\cdot v)\subset \conv(W\cdot v)$''. As for the reverse inclusion ``$\pi_\Sigma(G\cdot v)\supset \conv(W\cdot v)$''  (or, equivalently, convexity of $\pi_\Sigma(G\cdot v)$), its direct generalization ``$\pi_\Sigma(F)\supset \conv(F\cap\Sigma)$'' (or, equivalently, convexity of $\pi_\Sigma(F)$) fails in general,  see Section \ref{S:examples} for counter-examples. But we note that it does hold for {\bf isoparametric} foliations, see \cite{Terng86}. 

We finish with a couple of open questions:
\begin{question}
Is there a submetry with fat section, that is not isoparametric, and for which $\pi_\Sigma(F)= \conv(F\cap\Sigma)$ for every $\sigma$-fiber $F\subset V$?
\end{question}

\begin{question} \label{Q:spectrahedron}
In the situation of Corollary \ref{MC:isometry}, does the bijection between $\sigma$-saturated closed convex subsets of $V$ and $\sigma'$-saturated closed convex subsets of $V'$ preserve some special class of convex sets, such as spectrahedra or spectrahedral shadows?
\end{question}
Special cases of Question \ref{Q:spectrahedron} have received much attention recently  (see, for example, \cite{KS20,Kummer21, SS20, BKM24}).

\subsection*{Acknowledgements} It is a pleasure to thank Alexander Lytchak for help with Alexandrov Geometry, and for suggesting an earlier version of Theorem \ref{MT:quotient}. I am also grateful to Marco Radeschi for suggesting a proof (very similar to the one presented here) of Proposition \ref{P:saturation}(b), that convex hulls of saturated sets are saturated. Finally, I thank the anonymous referees for suggestions that have improved the exposition.

\section{Preliminaries}
\label{S:prelim}

\subsection{Alexandrov Geometry}
\label{SS:Alexandrov}
By an Alexandrov space we will mean a length space $X$ with lower curvature bound in the sense of Alexandrov, and finite dimension. See  \cite[Chapter 10]{BBI}, \cite{BGP92}, and \cite{Petrunin07}. For example, the space $X$ appearing in the statement of Theorems \ref{MT:quotient} and \ref{MT:fatsection} is a non-negatively curved Alexandrov space because it is the base of a submetry from the  non-negatively curved Alexandrov space $V$,  see \cite[4.6]{BGP92}.

Given a point $x$ in an Alexandrov space $X$, the tangent cone $T_xX$ can be defined as a limit of re-scalings at $x$. For a real-valued function $f$ defined near $x$, the differential $d_xf$ is a real-valued function on $T_xX$, also defined  as a limit of re-scalings, which is well-defined when $f$ is ``semi-concave'', see  \cite[Section 1.3]{Petrunin07}. (Semi-concave functions are assumed to be locally Lipschitz.)

In this article we will take $f$ to be a distance function to some subset $S$. That such $f$ is semi-concave on $X\setminus S$ follows from  the lower curvature bound.

A semi-concave function $f$ has a well-defined gradient, which is an element $\nabla_x f\in T_xX$ of the  tangent cone. Moreover, the norm $\|\nabla_x f\|$ of the gradient $\nabla_x f\in T_xX$ (that is, the distance to the apex of the cone $T_xX$) is the maximum of the differential $d_xf$ on the space of directions $\Sigma_x X\subset T_xX$, unless $d_xf\leq 0$, in which case $\|\nabla_x f\|=0$, see \cite[Section 1.3]{Petrunin07}. In particular,  $\|\nabla_x f\|$  is equal to the so-called ascending slope $|\nabla^+f|(x)$, defined by
\[ |\nabla^+f|(x)=\max \left\{ 0,\  \lim\sup_{y\to x} \frac{f(y)-f(x)}{d(y,x)} \right\}\]

When $f$ is a distance function, it is $1$-Lipschitz, and thus  $\|\nabla_x f\|\in[0,1]$.

\subsection{Submetries}
\label{SS:submetries}
The definition of \emph{submetry} goes back to \cite{Berestovskii87} (see also \cite{KL20}): A \emph{submetry} is a map $\sigma\colon  Y\to X$ between metric spaces that maps closed balls to closed balls of the same radius. The fibers of $\sigma$ (sometimes also called \emph{leaves}) form a decomposition of $Y$ into pairwise \emph{equidistant} closed subsets, in the sense that $d(x,F')=d(F,F')$ for every two fibers $F,F'$ and all $x\in F$. 

Conversely, given such a partition $\F$ of the metric space $Y$ (into ``leaves''), there is a unique metric on the set of leaves $X$ such that the natural map $Y\to X$ is a submetry. Endowed with this metric, $X$ is called the \emph{leaf space}.

A function on $Y$ is called \emph{$\sigma$-basic} (or just basic, if $\sigma$ is clear from context) if it is constant on the fibers of $\sigma$, that is, if it descends to a function on the ``base'' $X$. A subset of $Y$ is called \emph{$\sigma$-saturated} (or just saturated) if it is the union of $\sigma$-fibers, that is, if it is the inverse image under $\sigma$ of a subset of $X$.

The main source of examples of submetries are isometric group actions. Namely, if the group $G$ acts on the metric space $Y$ by isometries and with closed orbits, then the natural quotient map $Y\to Y/G$ is a submetry. The fibers (``leaves'') are the $G$-orbits, the saturated subsets of $Y$ are the $G$-invariant subsets, and the basic functions on $Y$ are the $G$-invariant functions. %If a submetry arises in this way from a group action, it is called ``homogeneous''.

In the present paper, we consider submetries $V\to X$, where $V$ will always denote a finite-dimensional real vector space with inner product (and associated Euclidean metric). We mention a few structure results that apply to this situation: $X$ is an Alexandrov space with non-negative curvature (see \cite[4.6]{BGP92}); every fiber is a subset of positive reach; and most fibers are $C^{1,1}$-submanifolds (see \cite{KL20} for the last two, and much more). If one adds the assumption that every fiber is a smooth submanifold, one arrives at the notion of ``manifold submetry'', see \cite{MR20, MR20apps} for structure results. 

We will many times add the assumption that the singleton $\{0\}$  is a $\sigma$-fiber, where $0\in V$ denotes the origin. It implies the following version of the Homothetic Transformation Lemma (see also \cite[Lemma 6.2]{Molino}, and \cite[Lemma 24]{MR20}). The importance of such submetries is that they model the ``infinitesimal behavior'' of more general submetries (compare \cite[Sections 5,7]{KL20}).
\begin{lemma}
\label{L:HTL}
Let $\sigma\colon V\to X$ be a submetry such that $\{0\}$ is a fiber. If $v,w\in V$ are such that $\sigma(v)=\sigma(w)$, then $\sigma(\lambda v)=\sigma(\lambda w)$ for all $\lambda\geq 0$.
\end{lemma}
\begin{proof}
The ray $t\mapsto tv$ (respectively $t\mapsto tw$), for $t\geq 0$, minimizes distance between $\{0\}$ and each fiber of the form $\sigma^{-1}(\sigma(t_0v))$ (respectively $\sigma^{-1}(\sigma(t_0w))$), for $t_0\geq 0$. Thus they descend to geodesic rays $\gamma_v, \gamma_w$ in $X$. Recall that, $X$ being an Alexandrov space, geodesics do not branch, and,  since $\gamma_v(0)=\gamma_w(0)$ and $\gamma_v(1)=\gamma_w(1)$, we obtain $\gamma_v=\gamma_w$ (see \cite[Exercises 10.1.2, 10.1.5]{BBI}). In particular, $\gamma_v(\lambda)=\gamma_w(\lambda)$, or, in other words, $\sigma(\lambda v)=\sigma (\lambda w)$.
\end{proof}

Given a submetry $\sigma\colon V\to X$, if the singleton $\{v\}$ is a fiber, we will say that $v$ is a \emph{fixed point} (of $\sigma$).

We will need the following Lemma, which follows from \cite[Proposition 5.6]{KL20}, and is also a slight generalization of a well-known fact about singular Riemannian foliations (see \cite[Proposition 5]{MR20quad}). For completeness we provide an elementary proof.
\begin{lemma}
\label{L:fixedpoints}
Let $\sigma\colon V\to X$ be a submetry such that the origin $0$ is a fixed point. Then the set $V_0$ of all fixed points is a vector subspace of $V$, and $\sigma$ ``splits'' in the sense that $X$ is isometric to $V_0\times X'$ for some metric space $X'$, and 
\[\sigma=\id_{V_0}\times\sigma'\colon  V=V_0\times V_0^\perp\to V_0\times X'=X.\]
Here $\sigma'\colon V_0^\perp\to X'$ is a submetry whose unique fixed point is the origin.
\end{lemma}

We first give a separate statement of a basic fact from Euclidean Geometry that will be useful in the proof of Lemma \ref{L:fixedpoints} and elsewhere:
\begin{lemma}
\label{L:Busemann}
Let $V$ be a vector space with inner product. Then, for $u,v\in V$ with $\|u\|=1$, one has:
\[\langle v, u\rangle=\sup_{t>0} \left(t-d\left(v,t u\right)\right)=\lim_{t\to\infty} \left(t-d\left(v,t u\right)\right)\]
\end{lemma}

\begin{proof}[Proof of Lemma \ref{L:fixedpoints}]
We use induction on $\dim(V)$. If $\dim(V)=0$, there is nothing to prove. Assume $\dim(V)>0$.

If the origin is the only fixed point, there is nothing to prove.

If $v\neq 0$ is a fixed point, consider the linear function $\lambda_v\colon V\to\R$ given by $\lambda_v(x)=\langle x, v\rangle$. We claim that $\lambda_v$ is basic. 

Indeed, assume $x,x'\in V\setminus \{0\}$ belong to the same fiber. Then $\|x\|=\|x'\|$ by equidistance of fibers (because the origin is a fiber), and so $\sigma(t x/\|x\|)=\sigma(t x'/\|x'\|)$ for all $t>0$ by the Homothetic Transformation Lemma (Lemma \ref{L:HTL}). Again by equidistance of fibers, we obtain $d\left(v,t x /\|x\|\right)=d\left(v,t x' /\|x'\|\right)$ for all $t>0$ (because $\{v\}$ is a fiber, by assumption). But Lemma \ref{L:Busemann} yields
\[\lambda_v(x)=\|x\| \sup_{t>0} \left(t-d\left(v,t x /\|x\|\right)\right)\]
and similarly for $x'$. Therefore $\lambda_v(x)=\lambda_v(x')$.

Since $\lambda_v$ is basic, the level sets are saturated. Given a fiber $F\subset \lambda_v^{-1}(0)$, it follows from equidistance of fibers that the translate $cv+ F$ is again a fiber, for every $c\in\R$ (see \cite[Prop. 5, proof of (b)$\implies$(a)]{MR20quad} for more details). Therefore $\sigma$ splits as $\id_\R\times\sigma'\colon \R\times \lambda_v^{-1}(0)\to X=\R\times X'$ for some submetry $\sigma'\colon  \lambda_v^{-1}(0)\to X'$. Applying the inductive hypothesis to $\sigma'$ finishes the proof.
\end{proof}

Regarding convex sets, we will use the elementary observation:
\begin{lemma}
\label{L:convexfixed}
Let $\sigma\colon V\to X$ be a submetry such that the origin $0$ is a fixed point. Then every closed convex $\sigma$-saturated set $K$ has a fixed point.
\end{lemma}
\begin{proof}
Take the point $v\in K$ closest to the origin (existence and uniqueness follow from the assumption that $K$ is closed and convex). Then $\{v\}$ is saturated, hence a fiber, because it is the intersection of two saturated sets: $K$ and the closed ball of radius $\|v\|$ around the origin.
\end{proof}

\subsection{Convex sets and support functions}
We will use the concepts of ``support function'' and ``polar'' (see e.g. \cite[pages 280--281]{SSS11}), as well as a couple of basic facts.
\begin{definition}
\label{D:support}
Let $A\subset V$ be a subset of a finite-dimensional real vector space with inner product. The \emph{support function} $h(A,\cdot)\colon V\to\R\cup\infty$ is defined by 
\[ h(A,v)=\sup_{a\in A} \langle v,a \rangle ,\]
and the \emph{polar} of $A$ is the subset $A^\circ\subset V$ defined by
\[A^\circ=\{ v\in V \mid h(A,v)\leq 1\}=\{ v\in V \mid \langle v ,a\rangle \leq 1 \quad \forall a\in A\}.\]
\end{definition}

\begin{proposition}
\label{P:bipolar}
Let $A\subset V$ be a closed subset of a  finite-dimensional real vector space with inner product.
\begin{enumerate}[(a)]
\item (Bipolar Theorem) If $\conv(A)$ contains the origin, then $A^{\circ\circ}=\conv(A)$.
\item If $\Sigma\subset V$ is a subspace, and $A$ is compact, convex, and contains the origin, then $\pi_\Sigma(A)=(\Sigma\cap A^\circ)^\circ$, where $\pi_\Sigma\colon V\to V$ denotes orthogonal projection onto $\Sigma$.
\end{enumerate}
\end{proposition}
\begin{proof}
\begin{enumerate}[(a)]
\item See \cite[IV (1.2) page 144]{Barvinok}.
\item It follows directly from Definition \ref{D:support} that the support function of $\pi_\Sigma(A)$ is the restriction to $\Sigma$ of the support function of $A$, and, thus, that $\Sigma\cap A^\circ =(\pi_\Sigma(A))^\circ$. 

Since $\pi_\Sigma(A)$ is a closed convex set containing the origin, the Bipolar Theorem implies that  $\pi_\Sigma(A)=(\pi_\Sigma(A))^{\circ\circ}=(\Sigma\cap A^\circ)^\circ$.
\end{enumerate}
\end{proof}

\section{Detecting convexity in the base of a submetry}
\label{S:detect}

In this section we provide a proof of  Theorem \ref{MT:quotient}, about how convexity can be detected metrically in the quotient, and we also investigate convex hulls of saturated sets. Some Alexandrov Geometry will be used, see Subsection \ref{SS:Alexandrov}, \cite[Chapter 10]{BBI}, \cite{BGP92}, and \cite{Petrunin07}.

%We begin with the following Lemma, which is a generalization of \cite[Example 6.6]{KL20}.
%\begin{lemma}
%\label{L:convexity}
%Let $V$ be a finite-dimensional real vector space with an inner product, $T\subset V$ be a closed subset, and  $\delta_T:\overline{V\setminus T}\to [0,\infty)$ the distance function to $T$.

%Then the following are equivalent:
%\begin{enumerate}[(a)]
%\item $T$ is a convex subset of $V$.
%\item $\|\nabla_x \delta_T\|=1$ for every $x\in V\setminus T$.
%\item $\delta_T$ is a submetry.
%\end{enumerate}
%\end{lemma}
%\begin{proof}
%
%First note that  $T$ is convex if and only if it has infinite reach, that is, for every  $x\in V\setminus T$, there is a unique minimizing geodesic between $x$ and $T$, see \cite[Theorems 6.2.12, 6.2.13]{KrantzParks}.
%
%Then  $(a)\iff (b)$ follows from \cite[Proposition 1.3 and Proposition 4.1]{KL21}, and $(a)\implies (c)$ is analogous to the proof of \cite[Proposition 6.3]{KL20}.
%
%Finally,  $(c) \implies (b)$ follows by setting $P=\delta_T$ and $g=\id_{[0,\infty)}$ in the formula in \cite[page 2662]{KL20}, and noting that, for a semi-concave function on an Alexandrov space, the ``ascending slope'' is nothing but the norm of its gradient.
%\end{proof}

\begin{proof}[Proof of Theorem \ref{MT:quotient}]
Let $T=\sigma^{-1}(S)$ and denote by $\delta_T:\overline{V\setminus T}\to [0,\infty)$ the distance function to $T$. Since $\sigma$ is a submetry, it follows that $\delta_T=\delta_S\circ \sigma$. By the top of page 2659 in \cite{KL20}, $\delta_T$ is a submetry if and only if $\delta_S$ is a submetry. Finally,  $\delta_T:\overline{V\setminus T}\to [0,\infty)$ is a submetry if and only if $T$ is convex (see \cite[Example 6.6]{KL20} for the special case where $T$ has empty interior --- the general case has a similar proof).\end{proof}

Whereas Theorem \ref{MT:quotient} states that convexity (which is equivalent to infinite reach, see \cite[Theorems 6.2.12, 6.2.13]{KrantzParks}) of a saturated set can be detected metrically in the quotient, the following analogous result applies to positive reach:
\begin{theorem}
\label{T:posreach}
Let $V$ be a finite-dimensional real vector space with an inner product, $X$ a metric space, and $\sigma\colon V\to X$ a submetry. Let $S\subset X$ be a closed subset, and denote by $\delta_S$
the distance function to $S$.

Then $\sigma^{-1}(S)$ is a set of positive reach in $V$ if and only if $\|\nabla_x \delta_S\|=1$ for every $x\in O\setminus S$, for some open subset $O\subset X$ containing $S$.
\end{theorem}
\begin{proof}
Let $T=\sigma^{-1}(S)$. First note that, for a semi-concave function on an Alexandrov space, the ``ascending slope'' is nothing but the norm of its gradient. Thus, setting $P=\sigma$ and $g=\delta_S$ in the formula in \cite[page 2662]{KL20}, it follows that $\|\nabla \delta_T\|=1$ on $\tilde{O}\setminus T$ for some open $\tilde{O}$ containing $T$ if and only if $\|\nabla \delta_S\|=1$ on $O\setminus S$ for some open $O$ containing $S$. The former condition is equivalent to $T$ having positive reach, by \cite[Proposition 1.3 and Proposition 4.1]{KL21}.
\end{proof}

\begin{remark} The condition $\|\nabla_x \delta_T\|=1$ for $x$ in an open set, which appears in the proof above, is also equivalent to differentiability of the distance function $\delta_T$, compare with \cite[Prop. 4.8]{Sakai} and \cite[Proposition 4.1]{KL21}.
\end{remark}

%
%\begin{remark}
%More generally, the distance function to the closed subset $\sigma^{-1}(S)\subset V$ has gradient of norm one \emph{near} $\sigma^{-1}(S)$ (as opposed to on all of $V\setminus \sigma^{-1}(S)$) if and only if $\sigma^{-1}(S)\subset V$ has ``positive reach'' --- see \cite[Proposition 1.3 and 4.1]{KL21}.
%
%Thus, one obtains an analogue of Theorem \ref{MT:quotient} for sets of positive reach (instead of convex). Namely, if one replaces ``for every $x\in X\setminus S$'' with ``for every $x\in O\setminus S$, for some open subset $O\subset X$ containing $S$'', one obtains a metric condition on $S$ equivalent to $\sigma^{-1}(S)$ having positive reach.
%\end{remark}

\subsection{Convex hulls and polars}

Using support functions (Definition \ref{D:support}), we show that taking the convex hull (or polar) preserves saturation:
\begin{proposition}
\label{P:saturation}
Given a submetry $\sigma\colon  V\to X$, such that $\{0\}$ is a fiber, and $A$ any closed saturated subset of $V$, one has:
\begin{enumerate}[(a)]
\item The support function $h(A,\cdot)$ is $\sigma$-basic.
\item The convex hull $\conv(A)$ is $\sigma$-saturated.
\item The polar $A^\circ$ is $\sigma$-saturated.
\end{enumerate}
\end{proposition}
\begin{proof}
\begin{enumerate}[(a)]
\item Assume first that $A$ is a single fiber $F$. Then $F$ is contained in a sphere of radius $r$ centered at the origin, and, using Lemma \ref{L:Busemann}, one obtains, for every $v\in V$:
\[h(F,v)=\sup_{a\in F} r\left(\sup_{t>0} \left( t - d\left(v,{t \over r}a\right) \right)\right)=r\sup_{t>0} \left( t - d\left(v,{t \over r}F\right) \right). \]
By the Homothetic Transformation Lemma (Lemma \ref{L:HTL}), the set ${t \over r}F$ is a fiber for every $t>0$, thus $d\left(\cdot,{t \over r}F\right)$ is a basic function for all $t>0$. Therefore $h(F,\cdot)$ is a basic function.

For $A$ not necessarily a single fiber, one has $h(A,v)=\sup_{F\subset A}h(F,v)$ for all $v\in V$, where the supremum is taken over all fibers $F$ contained in $A$. Thus $h(A,\cdot)$ is basic.
\item Since $A$ is closed, $\conv(A)$ is the intersection of all half-spaces that contain $A$. In terms of support functions, this reads
\[\conv(A)=\{x\in V\mid \langle v, x\rangle \leq \lambda \quad \forall v,\lambda \text{ such that } h(A,v)\leq \lambda \}.\]
By part (a), the support function $h(A, \cdot)$ is basic, and therefore this can be rewritten as 
\[\conv(A)=\left\{x\in V\mid \sup_{w\in F_v}\langle w, x\rangle \leq \lambda \quad \forall v,\lambda \text{ such that } h(A,v)\leq \lambda \right\},\]
where $F_v$ denotes the $\sigma$-fiber containing $v$. But $\sup_{w\in F_v}\langle w, x\rangle$ is exactly $h(L_v, x)$, and, again by part (a), $h(L_v, \cdot)$ is basic. Thus $\conv(A)$ is the intersection of saturated sets, hence saturated.
\item This follows immediately from part (a) and the definition of the polar as a sub-level set of the support function.
\end{enumerate}

\end{proof}

The next result says that convex hulls can be detected metrically in the quotient:
\begin{corollary}
\label{C:convexhull}
Consider a submetry $\sigma\colon  V\to X$, such that $\{0\}$ is a fiber. Denote by $\mathcal{C}$ the set of all closed subsets of $X$. Then, there exists a map $\conv_0: \mathcal{C}\to\mathcal{C}$, depending only on the metric structure of $X$, such that, for every $S\in\mathcal{C}$,  $\conv_0(S)$ is  the unique closed subset $C$ such that 
\[\overline{\conv(\sigma^{-1}(S))}=\sigma^{-1}(C),\]
 where $\conv(\cdot)$ denotes the convex hull of a subset of $V$.
\end{corollary}
\begin{proof}
For $S\in\mathcal{C}$, define $\conv_0(S)$ to be the intersection $C$ of all closed subsets of $X$ which contain $S$ and which satisfy the condition in Theorem \ref{MT:quotient}. By definition, $\conv_0$ depends only on the distance function of $X$. By Theorem \ref{MT:quotient}, $\sigma^{-1}(C)$ is the intersection of all closed, convex, saturated subsets of $V$ that contain $\sigma^{-1}(S)$. By Proposition \ref{P:saturation}(b), $\conv(\sigma^{-1}(S))$ is saturated, hence its closure is also saturated, and therefore $\overline{\conv(\sigma^{-1}(S))}=\sigma^{-1}(C)$.
\end{proof}

Generalizing ``orbitopes'', see e.g. \cite{SSS11}, we define a ``fibertope''  to be a convex set of the form $\conv(F)$, where $F$ is a single $\sigma$-fiber.
%, where $\sigma:V\to X$ is a submetry. 
When $\{0\}$ is a fiber of $\sigma$, Proposition \ref{P:saturation}(b) implies that fibertopes are saturated. The next result says that the bijection between saturated convex sets in Corollary \ref{MC:isometry} preserves fibertopes.

\begin{corollary}
\label{C:fibertope}
Let $V, V'$ be  finite-dimensional real vector spaces with inner product, $X,X'$  metric spaces, and $\sigma\colon V\to X$ and $\sigma'\colon V'\to X'$ be submetries. Assume the origins of $V,V'$ are fibers. Let $\varphi\colon X\to X'$ be an isometry sending $\sigma(0)$ to $\sigma'(0)$. Then $\varphi$ takes $\sigma$-images of fibertopes to $\sigma'$-images of fibertopes.
\end{corollary}
\begin{proof}
Since the maps $\conv_0(\cdot)$ given by Corollary \ref{C:convexhull} depend only on the metric structures, we have $\conv_0\circ\varphi=\varphi\circ\conv_0$. Let $K=\conv(F)$ be  a fibertope, where $F=\sigma^{-1}(p)$. Note that $K$ is compact, because $F$ is compact. We have $\sigma(K)=\conv_0(\{p\})$, and $\varphi(\conv_0(\{p\}))=\conv_0(\{\varphi(p)\})$. Thus $\varphi(\sigma(K))$ is the $\sigma'$-image of the fibertope $K'=\conv(F')$, for $F'=\sigma^{-1}(\varphi(p))$.
\end{proof}

\section{Fat sections}
\label{S:fatsections}

\begin{definition}
\label{D:fat}
Given a submetry $\sigma\colon  V\to X$, such that $\{0\}$ is a fiber, we call a subspace $\Sigma\subset V$ a \emph{fat section} for $\sigma$ if $\sigma|_\Sigma\colon \Sigma\to X$ is a submetry.
\end{definition}

In terms of equidistant decompositions (see Subsection \ref{SS:submetries}), Definition \ref{D:fat} can be rephrased as the following conditions on $\Sigma$: it meets all fibers, the decomposition $\F$ of $\Sigma$ given by $\F=\{F\cap\Sigma\mid F \text{ fiber of }\sigma \}$ is equidistant, and the natural bijection $\Sigma/\F\to X$ is an isometry.

In particular, the following are examples of fat sections: 
\begin{itemize}
\item $\sigma\colon  V\to X=V/G$ is the natural quotient map, where $V$ is a polar $G$-representation, and $\Sigma$ is a section. \todo{reference?}
\item More generally, $\sigma\colon  V\to V/G$, where $V$ is a $G$-representation of nontrivial ``copolarity'', and $\Sigma$ is a ``generalized minimal section''. See \cite{GOT04}, and \cite[Section 2.3]{GL14}).
\item $\sigma\colon  V\to V/G$, where $V$ is an effective $G$-representation with non-trivial principal isotropy group $K$, and $\Sigma$ is the fixed-point set $\Sigma=V^K$. See  \cite[page 62]{GL14} and references therein.
\item the fibers of $\sigma\colon V\to X$ form an isoparametric foliation, and $\Sigma$ is the normal space to a regular leaf. In this case, $\dim(\Sigma)=2$ and $\sigma|_\Sigma \colon  \Sigma\to X$ is the quotient map of a dihedral group action. \todo{reference?}
\end{itemize}

%Note that this is equivalent to $\Sigma$ meeting all fibers, and containing the normal space to every ``regular'' or ``principal'' fiber. \todo{make it more precise. need \emph{manifold} submetry for this to even make sense?}

%\begin{proposition} Let $A\subset V$ closed, and $\sigma\colon V\to X$ submetry such that $\{0\}$ is a fiber.
%\begin{enumerate}[(a)]
%\item If $\conv(A)$ contains the origin, then $A^{\circ\circ}=\conv(A)$.
%\item If $\Sigma$ is a subspace, and $A$ is convex, then $\pi_\Sigma(A)=(\Sigma\cap A^\circ)^\circ$.
%\item If $A$ is saturated, then so is $A^\circ$.
%\item If $F$ is a fiber, then $\conv(F)$ is saturated.
%\end{enumerate}
%\end{proposition}
%\begin{proof}
%(Sketch)
%
%(a), (b) are well-known.
%
%For (c), use the fact that
%\[\langle v, a\rangle=\|a\| \sup_{t\in\R} \left(t-d\left(v,t \frac{a}{\|a\|}\right)\right)\]
%and the Homothetic Transformation Lemma (which is available because $\{0\}$ is a fiber) to show that the support function of $A$ is basic. The polar is a sub-level set of the support function, hence saturated.
%
%(d) follows from (a), (c).
%\end{proof}

We turn to the proof of Theorem \ref{MT:fatsection}, which is a partial generalization of Theorem \ref{T:Kostant} (Kostant's Theorem), and a full generalization of Corollary \ref{C:Kostant}, to the case of fat sections. We note that Theorem \ref{T:Kostant} does not fully generalize to fat sections, see Section \ref{S:examples} for counter-examples.

Since Theorem \ref{MT:fatsection} concerns orthogonal projection onto a section, but Corollary \ref{MC:isometry} concerns intersection with the fat section, an important step is to link these two via Proposition \ref{P:bipolar} (including the Bipolar Theorem). A technical issue arises from the fact that Proposition \ref{P:bipolar} applies only to convex sets \emph{containing the origin}, which we address using Lemmas \ref{L:fixedpoints} and \ref{L:convexfixed}, together with the following:
\begin{lemma}
\label{L:fatfixed}
Let $\sigma\colon V\to X$ a submetry with $\{0\}$ a fiber, and $\Sigma\subset V$ a fat section. Denote by $V_0\subset V$ (respectively $\Sigma_0\subset \Sigma$) the set of all fixed points, that is, the set of $v\in V$ (respectively $v\in\Sigma$) such that $\{v\}$ is a $\sigma$-fiber (respectively $\sigma|_\Sigma$-fiber). Then $V_0=\Sigma_0$.
\end{lemma}
\begin{proof}
Since $\sigma|_\Sigma$ is onto $X$, one has $V_0\subset \Sigma_0$.

For the reverse inclusion $\Sigma_0\subset V_0$, use Corollary \ref{MC:isometry}. Indeed, for every $v\in\Sigma_0$, $\sigma|_\Sigma^{-1}(\sigma(v))=\{v\}$ is convex, hence $\sigma^{-1}(\sigma(v))$ is also convex, but since it is contained in a sphere, it must be the singleton $\{v\}$. In other words, $v\in V_0$.
\end{proof}

%The next lemma essentially corresponds to Theorem \ref{MT:fatsection}(b), in the special case where the convex saturated set $K\subset V$ is a ``fibertope'', that is, the convex hull of a single fiber. It will be used later to prove Theorem \ref{MT:fatsection}(b) for general $K$.
%\begin{lemma}
%\label{L:fibertopes}
%Let $\sigma:V\to X$ be a submetry such that $\{0\}$ is a fiber, $\Sigma\subset V$ be a fat section, $F\subset V$ be a fiber, and let $K=\conv(F)$. Then $K\cap\Sigma=\pi_\Sigma(K)=\conv(F\cap\Sigma)$. (Here $\pi_\Sigma:V\to\Sigma$ denotes the orthogonal projection onto $\Sigma$.)
%\end{lemma}
%\begin{proof}

%Using the previous lemma, the fact that closed convex saturated sets always contain a fixed point, and that translations in the direction of fixed points commute with both $\Sigma\cap\cdot$ and $\pi_\Sigma$, one may assume that $0\in K$.

%\end{proof}

\begin{proof}[Proof of Theorem \ref{MT:fatsection}]
\ \\ \vspace{-10pt}
\begin{enumerate}[(a)]
\item
First note that $F$ is compact (being a closed subset of a sphere), so $K=\conv(F)$ is a compact, convex, $\sigma$-saturated subset of $V$, by Proposition \ref{P:saturation}.

{\bf Reduction:} We reduce the proof to the case where $K$ contains the origin. 

Indeed, by Lemma \ref{L:convexfixed}, $K$ contains a fixed point $v\in V$, that is, a point $v$ such that $\{v\}$ is a $\sigma$-fiber. By Lemma \ref{L:fatfixed}, $v\in\Sigma$. 
%By Lemma \ref{L:fixedpoints}, the translation map $V\to V$ (respectively $\Sigma\to\Sigma$), given by $w\mapsto w-v$, sends $\sigma$-fibers (respectively $\sigma|_\Sigma$-fibers) to fibers.
By Lemma \ref{L:fixedpoints}, the translation map $V\to V$, given by $w\mapsto w-v$, sends $\sigma$-fibers to $\sigma$-fibers.

Thus $F-v$ is a fiber such that $K-v=\conv(F-v)$ contains the origin, and, assuming $(K-v)\cap\Sigma=\pi_\Sigma(K-v)=\conv((F-v)\cap\Sigma)$, we obtain
\[K\cap\Sigma=v+ (K-v)\cap\Sigma= v+ \pi_\Sigma(K-v)= \pi_\Sigma(K)\]
and, similarly, $K\cap\Sigma=\conv(F\cap\Sigma)$. This finishes the proof of the {\bf Reduction}.

By Corollary \ref{MC:isometry}, the map $A\mapsto \Sigma\cap A$ is a bijection between origin-containing saturated closed convex subsets of $V$ and $\Sigma$. This bijection preserves the partial order by inclusion.

By Propositions \ref{P:saturation}(c) and \ref{P:bipolar}(a), taking the polar is an order-reversing involution of the set of all origin-containing saturated closed convex subsets of $V$ (respectively $\Sigma$). 

Composing these bijections, the map $A\mapsto (\Sigma\cap A^\circ)^\circ$ is another order-preserving bijection between origin-containing saturated closed convex subsets of $V$ and $\Sigma$. If $A\subset V$ is a closed ball centered at the origin, $(\Sigma\cap A^\circ)^\circ$ is closed ball (in $\Sigma$) of the same radius. Thus, the map $A\mapsto (\Sigma\cap A^\circ)^\circ$ is also an order-preserving bijection between origin-containing saturated \emph{compact} convex subsets of $V$ and $\Sigma$. By Proposition \ref{P:bipolar}(b), this map coincides with $A\mapsto \pi_\Sigma(A)$.

Thus there is $K'\subset V$ compact, convex, saturated, origin-containing, such that $\pi_\Sigma(K')=K\cap\Sigma$. Since $K\cap\Sigma\subset \pi_\Sigma(K)$, we have $K'\subset K$.

Choose any $v\in\Sigma\cap F\subset K\cap\Sigma$, and $v'\in K'$ such that $\pi_\Sigma(v')=v$. Since $v'\in K=\conv(F)$, and $F$ is contained in the sphere of radius $\|v\|$ around the origin, we have $\| v' \|\leq \|v\|$. Thus $v=v'$, and $v\in K'$. Since $K'$ is saturated, we obtain $F\subset K'$, and, since $K'$ is convex, we obtain $K=\conv(F)\subset K'$. Thus $K=K'$, and $K\cap\Sigma=\pi_\Sigma(K)$.

The other equation  $K\cap\Sigma=\conv(F\cap\Sigma)$ follows from the fact that $A\mapsto A\cap\Sigma$ preserves ``fibertopes'', see Corollary \ref{C:fibertope}.

Finally, the last statement is clear: $\pi_\Sigma(F)\subset\pi_\Sigma(\conv(F))=\conv(F\cap\Sigma)$.
\item The inclusion $K\cap\Sigma\subset \pi_\Sigma(K)$ is clear. For the reverse inclusion, let $v\in\pi_\Sigma(K)$. Choose $v'\in K$ with $\pi_\Sigma(v')=v$, and denote by $F_{v'}\subset K$ the $\sigma$-fiber through $v'$. Then, by (a), $v\in\pi_\Sigma(F_{v'})\subset\conv(\Sigma\cap F_{v'})\subset K\cap\Sigma$.
%\item Clear from Corollary \ref{MC:isometry} and part (b) above.
\item The map $A\mapsto A\cap \Sigma$ is a bijection from the set of all $\sigma$-saturated subsets of $V$ to the set of all $\sigma|_\Sigma$-saturated subsets of $\Sigma$, with inverse $B\mapsto \sigma^{-1}(\sigma(B))$. We need to show that both these maps preserve convexity. The first map $A\mapsto A\cap \Sigma$ clearly does. Let $B\subset \Sigma$ be convex and $\sigma|_\Sigma$-saturated. Define $K=\conv(\sigma^{-1}(\sigma(B)))$, which is a convex $\sigma$-saturated set by Proposition \ref{P:saturation}(b). Using part (b), $K\cap\Sigma=\pi_\Sigma(K)=\conv(\pi_\Sigma(\sigma^{-1}(\sigma(B))))$. Using part (a) and convexity of $B$, for each $\sigma$-fiber $F\subset \sigma^{-1}(\sigma(B))$, we have $\pi_\Sigma(F)\subset\conv(F\cap\Sigma)\subset B$, and thus $K\cap\Sigma\subset B$. The reverse inclusion being clear, we obtain $K\cap\Sigma= B$, and therefore, $K=\sigma^{-1}(\sigma(B))$, that is, $\sigma^{-1}(\sigma(B))$ is convex.
\end{enumerate}
\end{proof}

\section{Examples}
\label{S:examples}

%Consequences of Theorem \ref{T:quotient}:

%\begin{corollary}
%Let $\sigma:V\to X$ and $\sigma':V'\to X'$ be submetries, and $\varphi:X\to X'$ an isometry. Then, for every closed $S\subset X$, we have that $\sigma^{-1}(S)$ is convex if and only if $(\sigma')^{-1}(\varphi(S))$ is convex.
%\end{corollary}
%
%\begin{corollary}
%\label{C:bij}
%Let $\sigma:V\to X$ be a submetry, and $\Sigma\subset V$ a subspace such that $\sigma|_\Sigma:\Sigma\to X$ is a submetry. Then intersection with $\Sigma$ gives a bijection between $\sigma$-saturated closed convex subsets of $V$ and $\sigma|_\Sigma$-saturated closed convex subsets of $\Sigma$.
%\end{corollary}

To illustrate that Theorem \ref{MT:fatsection} applies more generally than Kostant's Theorem, one can point to any representation that is not polar but has non-trivial copolarity in the sense of \cite{GOT04}. Here is one concrete example (see tables in \cite{GKW21} for many more):
\begin{example}
\label{E:copolarity}
Let $2\leq k\leq n-1$ be integers, $G=\O(n)$ be the orthogonal group\footnote{The group $\O(n)$ can be replaced with $\SO(n)$ and the orbits would not change. In other words, the orbits are connected.}, acting diagonally on $V=(\R^n)^{\oplus k}$. Then a principal isotropy is $H=\O(n-k)$ (embedded as a diagonal block), whose fixed point set is $\Sigma=(\R^k)^{\oplus k}$. The normalizer of $H$ in $G$ is $N(H)=\O(k)\times\O(n-k)$ (embedded as block-diagonal matrices). Its action has $\O(n-k)$ as ineffective kernel, so it has the same orbits as the diagonal action of $\O(k)$ on $\Sigma$. The orbit spaces $V/G$ and $\Sigma /N=\Sigma / \O(k)$ are isometric, and $\Sigma$ is a fat section for the natural quotient map $\sigma\colon V\to V/G$.
\end{example}

%\begin{example}
%\label{E:complexsymmetric}
%Let $V$ denote the space of all $n\times n$ symmetric complex matrices, $G=\SU(n)$ acting on $V$ by $(g,A)\mapsto gAg^t$. This representation is not polar. 
%
%The space $\Sigma$ of all diagonal matrices meets all $G$-orbits, and moreover $\Sigma$ contains the normal space to any principal $G$-orbit through $v\in\Sigma$. The normalizer $N\subset G$ of $\Sigma$ acts on $\Sigma$ with orbits equal to the the intersections of $G$-orbits with $\Sigma$. Thus the inclusion $\Sigma\to V$ induces an isometry $\Sigma /N \to V/G$, and Theorem \ref{MT:fatsection} gives a bijection between closed $G$-invariant convex subsets of $V$ and closed $N$-invariant convex subsets of $\Sigma$.
%\end{example}

Next we illustrate that Theorem \ref{T:Kostant} does not apply for all submetries with a fat section. That is, the reverse inclusion in Theorem \ref{MT:fatsection}(a) does not always hold.

The easiest counter-examples can be found considering polar representations with \emph{disconnected} orbits, thus also illustrating the necessity of the hypothesis that orbits are connected in Theorem \ref{T:Kostant}. At the most extreme, one can take $G\subset\O(V)$ finite and non-trivial. Then $\Sigma=V$ is a section, so $\pi_\Sigma$ is the identity map, and $\pi_\Sigma(G\cdot v)=G\cdot v$ is not convex unless it is a single point. Similar counter-examples with $G$ infinite are also easily constructed.

A more interesting example, with connected orbits:
\begin{example}
Let $V, G, \Sigma, \ldots$ as in Example \ref{E:copolarity}. Assume $k>\frac{2n-1}{3}$. Then, for a generic $v\in \Sigma$, $\pi_\Sigma(G\cdot v)$ is strictly contained in the orbitope $\O(k)\cdot v $. Indeed, since the $\O(k)$-representation $(\R^k)^{\oplus k}$ is of real type, the generic orbitope has non-empty interior, that is, dimension $k^2$ (see \cite[pages 279-280]{SSS11}). On the other hand, the principal $G$-orbits have dimension $\dim\O(n)-\dim\O(n-k)=kn-k(k+1)/2$. When $k>\frac{2n-1}{3}$, this is smaller than $k^2$.
\end{example}

\todo{add some inhomogeneous examples?}
%\begin{example}
%Relaxing to higher copolarity, but retaining connected orbits, one may obtain a counter-example by putting $n=2$ in Example \ref{E:complexsymmetric}: the orbit through $A=\diag(1,2)$ is $3$-dimensional, hence its projection onto $\Sigma$ (which is $4$-dimensional) has empty interior, so it can only be convex if it is contained in a hyperplane. But it is easy to produce enough affinely independent points in $\pi_\Sigma(G\cdot A)$, thus $\pi_\Sigma(G\cdot A)$ is not convex.
%\end{example}

\subsection*{Conflict of Interest}
The author has no conflict of interest to declare that are relevant to this article.
\subsection*{Data availability statement}
Not applicable.
\bibliography{ref}

\begin{thebibliography}{GLLM23}

\bibitem[Bar02]{Barvinok}
Alexander Barvinok.
\newblock {\em A course in convexity}, volume~54 of {\em Graduate Studies in
  Mathematics}.
\newblock American Mathematical Society, Providence, RI, 2002.

\bibitem[BBI01]{BBI}
Dmitri Burago, Yuri Burago, and Sergei Ivanov.
\newblock {\em A course in metric geometry}, volume~33 of {\em Graduate Studies
  in Mathematics}.
\newblock American Mathematical Society, Providence, RI, 2001.

\bibitem[Ber87]{Berestovskii87}
V.~N. Berestovski\u{\i}.
\newblock ``{S}ubmetries'' of three-dimensional forms of nonnegative curvature.
\newblock {\em Sibirsk. Mat. Zh.}, 28(4):44--56, 224, 1987.

\bibitem[BGP92]{BGP92}
Yu. Burago, M.~Gromov, and G.~Perel'man.
\newblock A. {D}. {A}leksandrov spaces with curvatures bounded below.
\newblock {\em Uspekhi Mat. Nauk}, 47(2(284)):3--51, 222, 1992.

\bibitem[BKM24]{BKM24}
Renato~G. {Bettiol}, Mario {Kummer}, and Ricardo A.~E. {Mendes}.
\newblock {Two results on the Convex Algebraic Geometry of sets with continuous
  symmetries}.
\newblock {\em arXiv e-prints}, page arXiv:2408.03231, August 2024.

\bibitem[Dad85]{Dadok85}
Jiri Dadok.
\newblock Polar coordinates induced by actions of compact {L}ie groups.
\newblock {\em Trans. Amer. Math. Soc.}, 288(1):125--137, 1985.

\bibitem[GKW21]{GKW21}
Claudio {Gorodski}, Andreas {Kollross}, and Burkhard {Wilking}.
\newblock {Actions on positively curved manifolds and boundary in the orbit
  space}.
\newblock {\em arXiv e-prints}, page arXiv:2112.00513, December 2021.

\bibitem[GL14]{GL14}
Claudio Gorodski and Alexander Lytchak.
\newblock On orbit spaces of representations of compact {L}ie groups.
\newblock {\em J. Reine Angew. Math.}, 691:61--100, 2014.

\bibitem[GLLM23]{GLLM19}
Claudio Gorodski, Christian Lange, Alexander Lytchak, and Ricardo A.~E. Mendes.
\newblock A diameter gap for quotients of the unit sphere.
\newblock {\em J. Eur. Math. Soc. (JEMS)}, 25(9):3767--3793, 2023.

\bibitem[GOT04]{GOT04}
Claudio Gorodski, Carlos Olmos, and Ruy Tojeiro.
\newblock Copolarity of isometric actions.
\newblock {\em Trans. Amer. Math. Soc.}, 356(4):1585--1608, 2004.

\bibitem[Gro02]{Grove02}
Karsten Grove.
\newblock Geometry of, and via, symmetries.
\newblock In {\em Conformal, {R}iemannian and {L}agrangian geometry
  ({K}noxville, {TN}, 2000)}, volume~27 of {\em Univ. Lecture Ser.}, pages
  31--53. Amer. Math. Soc., Providence, RI, 2002.

\bibitem[HK89]{HK89}
Wu-Yi Hsiang and Bruce Kleiner.
\newblock On the topology of positively curved {$4$}-manifolds with symmetry.
\newblock {\em J. Differential Geom.}, 29(3):615--621, 1989.

\bibitem[KL21]{KL21}
Vitali Kapovitch and Alexander Lytchak.
\newblock Remarks on manifolds with two-sided curvature bounds.
\newblock {\em Anal. Geom. Metr. Spaces}, 9(1):53--64, 2021.

\bibitem[KL22]{KL20}
Vitali Kapovitch and Alexander Lytchak.
\newblock The structure of submetries.
\newblock {\em Geom. Topol.}, 26(6):2649--2711, 2022.

\bibitem[Kos73]{Kostant73}
Bertram Kostant.
\newblock On convexity, the {W}eyl group and the {I}wasawa decomposition.
\newblock {\em Ann. Sci. \'{E}cole Norm. Sup. (4)}, 6:413--455 (1974), 1973.

\bibitem[KP99]{KrantzParks}
Steven~G. Krantz and Harold~R. Parks.
\newblock {\em The geometry of domains in space}.
\newblock Birkh\"{a}user Advanced Texts: Basler Lehrb\"{u}cher. [Birkh\"{a}user
  Advanced Texts: Basel Textbooks]. Birkh\"{a}user Boston, Inc., Boston, MA,
  1999.

\bibitem[KS22]{KS20}
Tim Kobert and Claus Scheiderer.
\newblock Spectrahedral representation of polar orbitopes.
\newblock {\em Manuscripta Math.}, 169(1-2):185--208, 2022.

\bibitem[Kum21]{Kummer21}
Mario Kummer.
\newblock Spectral linear matrix inequalities.
\newblock {\em Advances in Mathematics}, 384:107749, 2021.

\bibitem[LRT99]{LRT99}
R.~S. Leite, T.~R.~W. Richa, and C.~Tomei.
\newblock Geometric proofs of some theorems of {S}chur-{H}orn type.
\newblock {\em Linear Algebra Appl.}, 286(1-3):149--173, 1999.

\bibitem[{Mag}09]{Magata09}
Frederick {Magata}.
\newblock {Reductions, Resolutions and the Copolarity of Isometric Group
  Actions}.
\newblock {\em arXiv e-prints}, page arXiv:0908.0183, August 2009.

\bibitem[Mag10]{Magata10}
Frederick Magata.
\newblock A general {W}eyl-type integration formula for isometric group
  actions.
\newblock {\em Transform. Groups}, 15(1):184--200, 2010.

\bibitem[Men21]{Mendes20}
Ricardo A.~E. Mendes.
\newblock Lifting isometries of orbit spaces.
\newblock {\em Bull. Lond. Math. Soc.}, 53(6):1621--1626, 2021.

\bibitem[Mol88]{Molino}
Pierre Molino.
\newblock {\em Riemannian foliations}, volume~73 of {\em Progress in
  Mathematics}.
\newblock Birkh\"{a}user Boston, Inc., Boston, MA, 1988.
\newblock Translated from the French by Grant Cairns, With appendices by
  Cairns, Y. Carri\`ere, \'{E}. Ghys, E. Salem and V. Sergiescu.

\bibitem[MR20a]{MR20quad}
R.~A.~E. Mendes and M.~Radeschi.
\newblock Singular {R}iemannian foliations and their quadratic basic
  polynomials.
\newblock {\em Transform. Groups}, 25(1):251--277, 2020.

\bibitem[MR20b]{MR20}
Ricardo A.~E. Mendes and Marco Radeschi.
\newblock Laplacian algebras, manifold submetries and the inverse invariant
  theory problem.
\newblock {\em Geom. Funct. Anal.}, 30(2):536--573, 2020.

\bibitem[MR23]{MR20apps}
Ricardo A.~E. Mendes and Marco Radeschi.
\newblock Maximality of {L}aplacian algebras, with applications to invariant
  theory.
\newblock {\em Ann. Mat. Pura Appl. (4)}, 202(2):1011--1031, 2023.

\bibitem[Pet07]{Petrunin07}
Anton Petrunin.
\newblock Semiconcave functions in {A}lexandrov's geometry.
\newblock In {\em Surveys in differential geometry. {V}ol. {XI}}, volume~11 of
  {\em Surv. Differ. Geom.}, pages 137--201. Int. Press, Somerville, MA, 2007.

\bibitem[Sak96]{Sakai}
Takashi Sakai.
\newblock {\em Riemannian geometry}, volume 149 of {\em Translations of
  Mathematical Monographs}.
\newblock American Mathematical Society, Providence, RI, 1996.
\newblock Translated from the 1992 Japanese original by the author.

\bibitem[SS20]{SS20}
Raman {Sanyal} and James {Saunderson}.
\newblock {Spectral Polyhedra}.
\newblock {\em arXiv e-prints}, page arXiv:2001.04361, January 2020.

\bibitem[SSS11]{SSS11}
Raman Sanyal, Frank Sottile, and Bernd Sturmfels.
\newblock Orbitopes.
\newblock {\em Mathematika}, 57(2):275--314, 2011.

\bibitem[Ter86]{Terng86}
Chuu-Lian Terng.
\newblock Convexity theorem for isoparametric submanifolds.
\newblock {\em Invent. Math.}, 85(3):487--492, 1986.

\end{thebibliography}
\bibliographystyle{alpha}

\end{document}